\documentclass [reqno] {amsart}

\usepackage{latexsym}
\usepackage{amssymb}
\usepackage{amsthm}
\usepackage{amscd}
\usepackage{amsmath}
\usepackage{tikz}

\usepgflibrary{arrows}

\newtheorem{theorem}{Theorem}[section]

\newtheorem{proposition}[theorem]{Proposition}
\newtheorem{corollary}[theorem]{Corollary}

\theoremstyle{definition}
\newtheorem{example}{Example}[section]
\newtheorem*{remark}{Remark}

\numberwithin{equation}{section}

\newcommand{\FF}{\mathbb{F}}

\newcommand{\CC}{\mathbb{C}}

\newcommand{\ZZ}{\mathbb{Z}}

\newcommand{\GL}{\mathrm{GL}} 
\newcommand{\Res}{\mathrm{Res}} 
\newcommand{\Inf}{\mathrm{Inf}}
\newcommand{\ch}{\mathrm{ch}}

\newcommand{\UT}{\mathrm{UT}}

\newcommand{\nst}{\mathrm{nst}}
\newcommand{\rnode}{\mathrm{rnode}}
\newcommand{\spanning}{\text{-span}}
\newcommand{\dimv}{\mathrm{dimv}}

\newcommand{\cX}{\mathcal{X}}
\newcommand{\cS}{\mathcal{S}}

\newcommand{\cK}{\mathcal{K}}

\newcommand{\NCSym}{\mathrm{NCSym}}

\newcommand{\SC}{\mathrm{SC}}

\newcommand{\dd}{\displaystyle}
 
\newcommand{\scs}{\scriptstyle}

\newcommand{\larc}[1]{\hspace{-.4ex}\overset{#1}{\frown}\hspace{-.4ex}}
\newcommand{\slarc}[1]{\overset{#1}{\frown}}
\def\adots{\mathinner{\mkern2mu\raise0pt\hbox{.}  % antidiagonal dots
\mkern2mu\raise4pt\hbox{.}\mkern1mu
\raise7pt\vbox{\kern7pt\hbox{.}}\mkern1mu}}

\makeatletter
\renewcommand{\@makefnmark}{\mbox{\textsuperscript{}}}
\makeatother

\allowdisplaybreaks[1]

\begin{document}

\title[A supercharacter table decomposition]
{A supercharacter table decomposition via\\ power-sum symmetric functions}

\author{N. Bergeron}\address[Nantel Bergeron]
{Department of Mathematics and Statistics\\ York  University\\ To\-ron\-to, Ontario M3J 1P3\\ CANADA}
\email{bergeron@mathstat.yorku.ca}
\urladdr{http://www.math.yorku.ca/bergeron}

 \author{N. Thiem}\address[Nathaniel Thiem]
 {University of Colorado \textbf{Boulder}}
 \email{Nathaniel.Thiem@Colorado.EDU}

\date{\today}

\thanks{N. Bergeron is supported in part by CRC and NSERC}
\thanks{N. Thiem is supported in part by NSF FRG DMS-0854893}

%%%%%%%%%%%%%%%%%%%%%%%%%%%%%%%%%%%
%%%%%%%%%%%%%%%%%%%%%%%%%%%%%%%%%%%
\begin{abstract}
 We give an $LU$-decomposition of the supercharacter table of the group of $n\times n$ unipotent upper triangular matrices over $\FF_q$, 
into a lower-triangular matrix with entries in $\ZZ[q]$ and an upper-triangular matrix with entries in 
$\ZZ[q^{-1}]$. To this end we introduce a $q$ deformation of a new power-sum basis of the Hopf algebra of symmetric functions in noncommuting variables. The decomposition is obtained from the transition matrices between the supercharacter basis, the $q$-power-sum basis and the superclass basis.
This is  similar to the decomposition of the character table of the symmetric group $S_n$ given by the transition matrices between Schur functions, monomials and power-sums.

We deduce some combinatorial results associated to this decomposition. In particular we compute the determinant of the supercharacter table.

\end{abstract}

\maketitle
\noindent{\it \`A mon ami Christophe Reutenauer et aux bons moments pass\'es ensemble.
}

%%%%%%%%%%%%%%%%%%%%%%%%%%%%%%%%%%%%%%%%%%%%%%%%%%%%%%%%%%%%%%%%%%
\section{Introduction}

It is well known (see \cite{Mcd}) that the representation theory of the symmetric groups is nicely encoded by the space of symmetric functions (in countably many commuting variables). In fact the interplay between the character theory of the symmetric groups and symmetric functions has enriched both theories with very interesting combinatorics. The space of symmetric functions has several algebraic operations (in particular it is a Hopf algebra) and many interesting bases (Schur, power-sum, monomial,  and homogeneous symmetric functions). 
The algebraic operations and bases can be lifted to the characters of the symmetric groups, and as such are meaningful representation theoretic operations and bases. The character table of the symmetric group is known to be the transition matrix between the Schur basis and the power-sum basis. A natural factorization of this matrix is obtained by using a third basis (the monomial basis). The transition matrix between Schur functions and monomials is unipotent lower triangular and the transition matrix between monomials and power-sums is upper triangular.

In a recent workshop at AIM \cite{AIM}, we showed that the supercharacter theory of the group of unipotent upper triangular matrices over a finite field $\FF_q$ is related to the Hopf algebra $\NCSym(X)$ of symmetric functions in noncommutative variables \cite{BRRZ, BZ,RS}. For $q=2$, the algebraic operations of $\NCSym(X)$ can be lifted to the supercharacter theory and have a representation theoretic meaning. 
This inspired us to seek a new basis of $\NCSym(X)$ that will allow a natural decomposition of the supercharacter table. 

To this end, we recall in Section~\ref{sec:superchar} some of the results of \cite{AIM}, and then adapt it to a coarser supercharacter theory that allows us to have an isomorphism to $\NCSym(X)$ valid for all $q$. The supercharacter table is given by the transition matrix between the supercharacter basis and the superclass basis.
In Section~\ref{sec:qpowersums} we introduce a $q$-deformation of a new power-sum basis (these power-sums were first introduced in \cite{ABT}). For each $q$, this will give us our desired factorization of the supercharacter table. In subsequent sections, we explore the sum of the consequences and related combinatorics. 
In particular we  compute the determinant of the character table.

%%%%%%%%%%%%%%%%%%%%%%%%%%%%%%%%%%%%%%%%%%%%%%%%%%%%%%%%%%%%%%%%%%
\section{Preliminaries}\label{sec:superchar}

%%%%%%%%%%%%%%%%%%%%%%%%%%%%%%%%%%%%%%%%%%%%%%%%%%%%%%%%%%%%%%%%%%
\subsection{Supercharacters}

A \emph{supercharacter theory} of a finite group $G$ is a pair $(\cK,\cX)$ where $\cK$ is a partition of $G$ such that 
$$\CC\spanning\bigg\{\sum_{g\in K}g\mid K\in \cK\bigg\}$$
is a dimension $|\cK|$ subalgebra of $Z(\CC G)$  under usual group algebra multiplication and $\cX$ is a partition of the irreducible characters of $G$ such that $|\cX|=|\cK|$ and 
\begin{equation}\label{KisX}
\SC(G)=\left\{f:G\rightarrow \CC\ \bigg|\ \begin{array}{@{}l@{}} f\text{ constant on}\\ \text{the parts of $\cK$}\end{array}\right\}=\CC\spanning\bigg\{ \sum_{\psi\in X} \psi(1)\psi \mid X\in \cX\bigg\}
\end{equation}

We will refer to the parts $K\in \cK$ as \emph{superclasses}; we fix a basis of $\SC(G)$ consisting of characters orthogonal with respect to the usual inner product on class functions, and refer to the elements of this basis as \emph{supercharacters}.

There are various natural supercharacter theories for the group
$$\UT_n(q)=\left\{\begin{array}{c}\text{$n\times n$ unipotent upper triangular }\\ \text{matrices over $\FF_q$}\end{array}\right\},$$
but for this paper, we are interested in the following theory.  Let $u,v\in \UT_n(q)$ be equivalent, if there exist $x,y\in \UT_n(q)$ and $t\in T_n(q)$ such that $u=xt(v-1)t^{-1}y+1$. Here  $T_n(q)\subseteq \GL_n(\FF_q)$ denotes the set of diagonal matrices with non-zero entries on the diagonal. We will let $\cK$ be the set of equivalence classes of this relation, giving half of our supercharacter theory. 

 It turns out that these superclasses are indexed by
 $$\cS_n=\{\text{set partitions of $\{1,2,\ldots, n\}$}\},$$
  where a \emph{set partition} $\lambda$ of $\{1,2,\ldots, n\}$ is a subset $\lambda\subseteq \{i\larc{}j\mid 1\leq i<j\leq n\}$ such that 
 $$i\larc{}k\in \lambda\qquad \text{implies}\qquad i\larc{}j,j\larc{}k\notin \lambda\qquad \text{for $i<j<k$.}$$

Instead of finding the corresponding partition $\cX$ of the irreducible characters of $\UT_n(q)$ (which is uniquely determined by $\cK$ via (\ref{KisX})), we will give our chosen set of supercharacters.  Note that since the supercharacters form a basis for $\SC(\UT_n(q))$, we have that they are also indexed by $\cS_n$.   Given $\lambda,\mu\in \cS_n$ with $u_\mu$ in the superclass corresponding to $\mu$, define $\chi^\lambda\in \SC(\UT_n(q))$ by
\begin{equation}\label{SupercharacterFormula}
\chi^\lambda(u_\mu)=\left\{\begin{array}{ll} \dd\frac{(q-1)^{|\lambda|-|\lambda\cap\mu|}q^{\dim(\lambda)-|\lambda|} (-1)^{|\lambda\cap\mu|}}{q^{\nst^\lambda_\mu}} & \begin{array}{@{}l} \text{if $i<j<k$ with $i\larc{}k\in \lambda$}\\  \text{implies $i\larc{}j,j\larc{}k\notin \mu$,}\end{array}\\  0 & \text{otherwise,}\end{array}\right.
\end{equation}
where 
\begin{align*}
\dim(\lambda) &= \sum_{i\slarc{} j\in \lambda} j-i,\\
\nst^\lambda_\mu &= \#\{i< j<k < l\mid i\larc{}l\in \lambda, j\larc{}k\in \mu\}.
\end{align*}

These superclass functions are characters, and they form a basis for $\SC(\UT_n(q))$ in this case.

\begin{remark}
The supercharacter theory defined in this paper is slightly coarser than the usual supercharacter theory used for $\UT_n(q)$ (for example, \cite{An95,DI}).  In the finer theory, we discard the conjugation action of $T$ in our equivalence relation.  However, these two supercharacter theories coincide when $q=2$.
\end{remark}

\begin{example}\label{SupercharacterTable}
For $n=4$, if $t=q-1$, then the supercharacter table is
$$
\begin{array}{c|ccccccccccccccc}
 & \begin{tikzpicture}[scale=.2]
	\foreach \x in {0,1,2,3}
		\node (\x) at (\x,0) [inner sep =-2pt] {$\scs\bullet$};
\end{tikzpicture} 
& 
\begin{tikzpicture}[scale=.2]
	\foreach \x in {0,1,2,3}
		\node (\x) at (\x,0) [inner sep =-2pt] {$\scs\bullet$};
	\draw (0) .. controls (.25,.75) and (.75,.75) .. (1);
\end{tikzpicture} &
\begin{tikzpicture}[scale=.2]
	\foreach \x in {0,1,2,3}
		\node (\x) at (\x,0) [inner sep =-2pt] {$\scs\bullet$};
	\draw (1) .. controls (1.25,.75) and (1.75,.75) .. (2);
\end{tikzpicture} & 
\begin{tikzpicture}[scale=.2]
	\foreach \x in {0,1,2,3}
		\node (\x) at (\x,0) [inner sep =-2pt] {$\scs\bullet$};
	\draw (2) .. controls (2.25,.75) and (2.75,.75) .. (3);
\end{tikzpicture} 
& 
\begin{tikzpicture}[scale=.2]
	\foreach \x in {0,1,2,3}
		\node (\x) at (\x,0) [inner sep =-2pt] {$\scs\bullet$};
	\draw (0) .. controls (0.25,.75) and (.75,.75) .. (1);
	\draw (1) .. controls (1.25,.75) and (1.75,.75) .. (2);
\end{tikzpicture} 
& 
\begin{tikzpicture}[scale=.2]
	\foreach \x in {0,1,2,3}
		\node (\x) at (\x,0) [inner sep =-2pt] {$\scs\bullet$};
	\draw (0) .. controls (0.25,.75) and (.75,.75) .. (1);
	\draw (2) .. controls (2.25,.75) and (2.75,.75) .. (3);
\end{tikzpicture} 
& 
\begin{tikzpicture}[scale=.2]
	\foreach \x in {0,1,2,3}
		\node (\x) at (\x,0) [inner sep =-2pt] {$\scs\bullet$};
	\draw (1) .. controls (1.25,.75) and (1.75,.75) .. (2);
	\draw (2) .. controls (2.25,.75) and (2.75,.75) .. (3);
\end{tikzpicture} 
& 
\begin{tikzpicture}[scale=.2]
	\foreach \x in {0,1,2,3}
		\node (\x) at (\x,0) [inner sep =-2pt] {$\scs\bullet$};
	\draw (0) .. controls (0.25,.75) and (.75,.75) .. (1);
	\draw (1) .. controls (1.25,.75) and (1.75,.75) .. (2);
	\draw (2) .. controls (2.25,.75) and (2.75,.75) .. (3);
\end{tikzpicture} 
&
\begin{tikzpicture}[scale=.2]
	\foreach \x in {0,1,2,3}
		\node (\x) at (\x,0) [inner sep =-2pt] {$\scs\bullet$};
	\draw (0) .. controls (.5,1.25) and (1.5,1.25) .. (2);
\end{tikzpicture} 
&
\begin{tikzpicture}[scale=.2]
	\foreach \x in {0,1,2,3}
		\node (\x) at (\x,0) [inner sep =-2pt] {$\scs\bullet$};
	\draw (1) .. controls (1.5,1.25) and (2.5,1.25) .. (3);
\end{tikzpicture} 
&
\begin{tikzpicture}[scale=.2]
	\foreach \x in {0,1,2,3}
		\node (\x) at (\x,0) [inner sep =-2pt] {$\scs\bullet$};
	\draw (0) .. controls (0.25,.75) and (.75,.75) .. (1);
	\draw (1) .. controls (1.5,1.25) and (2.5,1.25) .. (3);
\end{tikzpicture} 
&
\begin{tikzpicture}[scale=.2]
	\foreach \x in {0,1,2,3}
		\node (\x) at (\x,0) [inner sep =-2pt] {$\scs\bullet$};
	\draw (0) .. controls (.5,1.25) and (1.5,1.25) .. (2);
	\draw (2) .. controls (2.25,.75) and (2.75,.75) .. (3);
\end{tikzpicture} 
&
\begin{tikzpicture}[scale=.2]
	\foreach \x in {0,1,2,3}
		\node (\x) at (\x,0) [inner sep =-2pt] {$\scs\bullet$};
	\draw (0) .. controls (.5,1.25) and (1.5,1.25) .. (2);
	\draw (1) .. controls (1.5,1.25) and (2.5,1.25) .. (3);
\end{tikzpicture} 
&
\begin{tikzpicture}[scale=.2]
	\foreach \x in {0,1,2,3}
		\node (\x) at (\x,0) [inner sep =-2pt] {$\scs\bullet$};
	\draw (0) .. controls (.75,1.75) and (2.25,1.75) .. (3);
\end{tikzpicture} 
&
\begin{tikzpicture}[scale=.2]
	\foreach \x in {0,1,2,3}
		\node (\x) at (\x,0) [inner sep =-2pt] {$\scs\bullet$};
	\draw (0) .. controls (.75,1.75) and (2.25,1.75) .. (3);
	\draw (1) .. controls (1.25,.75) and (1.75,.75) .. (2);
\end{tikzpicture} \\ \hline
  \begin{tikzpicture}[scale=.2]
	\foreach \x in {0,1,2,3}
		\node (\x) at (\x,0) [inner sep =-2pt] {$\scs\bullet$};
\end{tikzpicture} 
& 1 & 1 & 1 & 1 & 1 & 1 & 1 & 1 & 1 & 1 & 1 & 1 & 1 & 1 & 1 \\
\begin{tikzpicture}[scale=.2]
	\foreach \x in {0,1,2,3}
		\node (\x) at (\x,0) [inner sep =-2pt] {$\scs\bullet$};
	\draw (0) .. controls (.25,.75) and (.75,.75) .. (1);
\end{tikzpicture} &
t & -1 & t & t & -1 & -1 & t & -1 & t & t & t & -1 &  t & t & t \\
\begin{tikzpicture}[scale=.2]
	\foreach \x in {0,1,2,3}
		\node (\x) at (\x,0) [inner sep =-2pt] {$\scs\bullet$};
	\draw (1) .. controls (1.25,.75) and (1.75,.75) .. (2);
\end{tikzpicture} &  t & t & -1 & t & -1 & t & -1 & -1 & t & t & t & t & t & t & -1\\
\begin{tikzpicture}[scale=.2]
	\foreach \x in {0,1,2,3}
		\node (\x) at (\x,0) [inner sep =-2pt] {$\scs\bullet$};
	\draw (2) .. controls (2.25,.75) and (2.75,.75) .. (3);
\end{tikzpicture} 
&  t & t & t & -1 & t & -1 & -1 & -1 & t & t & t & -1 & t & t & t\\
\begin{tikzpicture}[scale=.2]
	\foreach \x in {0,1,2,3}
		\node (\x) at (\x,0) [inner sep =-2pt] {$\scs\bullet$};
	\draw (0) .. controls (0.25,.75) and (.75,.75) .. (1);
	\draw (1) .. controls (1.25,.75) and (1.75,.75) .. (2);
\end{tikzpicture} 
&  t^2 & -t & -t & t^2 & 1 & -t & -t & 1 & t^2 & t^2 & -t & t^2 & t^2 & t^2 & -t\\ 
\begin{tikzpicture}[scale=.2]
	\foreach \x in {0,1,2,3}
		\node (\x) at (\x,0) [inner sep =-2pt] {$\scs\bullet$};
	\draw (0) .. controls (0.25,.75) and (.75,.75) .. (1);
	\draw (2) .. controls (2.25,.75) and (2.75,.75) .. (3);
\end{tikzpicture} 
& t^2 & -t & t^2 & -t & -t & 1 & -t & 1 & t^2 & t^2 & -t & -t & t^2 & t^2 & t^2  \\
\begin{tikzpicture}[scale=.2]
	\foreach \x in {0,1,2,3}
		\node (\x) at (\x,0) [inner sep =-2pt] {$\scs\bullet$};
	\draw (1) .. controls (1.25,.75) and (1.75,.75) .. (2);
	\draw (2) .. controls (2.25,.75) and (2.75,.75) .. (3);
\end{tikzpicture} 
& t^2 & t^2 & -t & -t & -t & -t & 1 & 1 & t^2 & t^2 & t^2 & -t & t^2 & t^2 & -t\\
\begin{tikzpicture}[scale=.2]
	\foreach \x in {0,1,2,3}
		\node (\x) at (\x,0) [inner sep =-2pt] {$\scs\bullet$};
	\draw (0) .. controls (0.25,.75) and (.75,.75) .. (1);
	\draw (1) .. controls (1.25,.75) and (1.75,.75) .. (2);
	\draw (2) .. controls (2.25,.75) and (2.75,.75) .. (3);
\end{tikzpicture} 
& t^3 & -t^2 & -t^2 & -t^2 & t & t & t & -1 & t^3 & t^3 & -t^2 & -t^2 & t^3 & t^3 & -t^2\\
\begin{tikzpicture}[scale=.2]
	\foreach \x in {0,1,2,3}
		\node (\x) at (\x,0) [inner sep =-2pt] {$\scs\bullet$};
	\draw (0) .. controls (.5,1.25) and (1.5,1.25) .. (2);
\end{tikzpicture} 
& tq & 0 & 0 & tq & 0 & 0 & 0 & 0 & -q & tq & 0 & -q & -q & tq & 0\\
\begin{tikzpicture}[scale=.2]
	\foreach \x in {0,1,2,3}
		\node (\x) at (\x,0) [inner sep =-2pt] {$\scs\bullet$};
	\draw (1) .. controls (1.5,1.25) and (2.5,1.25) .. (3);
\end{tikzpicture} 
& tq & tq & 0 & 0 & 0 & 0 & 0 & 0 & tq & -q & -q & 0 & -q & tq & 0\\
\begin{tikzpicture}[scale=.2]
	\foreach \x in {0,1,2,3}
		\node (\x) at (\x,0) [inner sep =-2pt] {$\scs\bullet$};
	\draw (0) .. controls (0.25,.75) and (.75,.75) .. (1);
	\draw (1) .. controls (1.5,1.25) and (2.5,1.25) .. (3);
\end{tikzpicture} 
& t^2q & -tq & 0 & 0 & 0 & 0 & 0 & 0 & t^2q & -tq & q & 0 & -tq & t^2q & 0\\
\begin{tikzpicture}[scale=.2]
	\foreach \x in {0,1,2,3}
		\node (\x) at (\x,0) [inner sep =-2pt] {$\scs\bullet$};
	\draw (0) .. controls (.5,1.25) and (1.5,1.25) .. (2);
	\draw (2) .. controls (2.25,.75) and (2.75,.75) .. (3);
\end{tikzpicture} 
& t^2q & 0 & 0 & -tq & 0 & 0 & 0 & 0 & -tq & t^2q & 0 & q & -tq & t^2 q & 0\\
\begin{tikzpicture}[scale=.2]
	\foreach \x in {0,1,2,3}
		\node (\x) at (\x,0) [inner sep =-2pt] {$\scs\bullet$};
	\draw (0) .. controls (.5,1.25) and (1.5,1.25) .. (2);
	\draw (1) .. controls (1.5,1.25) and (2.5,1.25) .. (3);
\end{tikzpicture} 
& t^2q^2 & 0 & 0 & 0 & 0 & 0 & 0 & 0 & -tq^2 & -tq^2 & 0 & 0 & q^2 & t^2q^2 & 0\\
\begin{tikzpicture}[scale=.2]
	\foreach \x in {0,1,2,3}
		\node (\x) at (\x,0) [inner sep =-2pt] {$\scs\bullet$};
	\draw (0) .. controls (.75,1.75) and (2.25,1.75) .. (3);
\end{tikzpicture} 
& tq^2 & 0 & tq & 0 & 0 & 0 & 0 & 0 & 0 & 0 & 0 & 0 & 0 & -q^2 & -q\\
\begin{tikzpicture}[scale=.2]
	\foreach \x in {0,1,2,3}
		\node (\x) at (\x,0) [inner sep =-2pt] {$\scs\bullet$};
	\draw (0) .. controls (.75,1.75) and (2.25,1.75) .. (3);
	\draw (1) .. controls (1.25,.75) and (1.75,.75) .. (2);
\end{tikzpicture}  & t^2 q^2 & 0 & -tq & 0 & 0 & 0 & 0 & 0 & 0 & 0 & 0 & 0 & 0 & -tq^2 & q
\end{array}
$$
\end{example}

%%%%%%%%%%%%%%%%%%%%%%%%%%%%%%%%%%%%%%%%%%%%%%%%%%%%%%%%%%%%%%%%%%
\subsection{Hopf algebra of supercharacters}

Let 
$$\SC(q)=\bigoplus_{n\geq 0} \SC(\UT_n(q)),$$
where by convention we let 
$$\SC(\UT_0(q))=\CC\spanning\{\chi^{\emptyset_0}\},$$
where $\emptyset_0$ is the empty set partition of the set with 0 elements.  Define a product on $\SC(q)$ by
$$\chi_m\cdot \chi_n= \Inf_{\UT_m(q)\times \UT_n(q)}^{\UT_{m+n}(q)}(\chi_m\times \chi_n) = (\chi_m\times \chi_n)\circ \pi,$$
where $\chi_m\in \SC(\UT_m(q))$, $\chi_n\in \SC(\UT_n(q))$, and $\Inf$ is the inflation functor coming from the quotient map
$$\pi\colon \UT_{m+n}(q)\longrightarrow \left[\begin{array}{c|c} \UT_m(q) & 0\\ \hline 0 & \UT_n(q)\end{array}\right] \cong  \UT_m(q) \times \UT_n(q).$$
Define a coproduct on $\SC(q)$ by
$$\Delta(\chi_n)=\sum_{\{1,2,\ldots,n\}=J\sqcup K} \Res^{\UT_n(q)}_{\UT_J(q)}(\chi_n)\otimes  \Res^{\UT_n(q)}_{\UT_K(q)}(\chi_n),$$
where $\UT_J(q)$ is the subgroup of $\UT_n(q)$ with nonzero entries above the diagonal only in rows and columns in $J$.  We make use of the isomorphism $\UT_J(q)\cong \UT_{|J|}(q)$ in this definition.

This product and coproduct give rise to a graded Hopf algebra, and this algebra comes equipped with two distinguished bases:
\begin{align*}
\SC(q) & = \CC\spanning\{\chi^\lambda\mid \lambda\in \cS_n, n\in \ZZ_{\geq 0}\}\\
& = \CC\spanning\{\kappa_\mu\mid \mu\in \cS_n, n\in \ZZ_{\geq 0}\},
\end{align*}
where for $u\in \UT_n(q)$,
$$\kappa_\mu(u)=\left\{\begin{array}{ll} 1 & \text{if $u$ is in the superclass indexed by $\mu$,}\\ 0  & \text{otherwise.}\end{array}\right.$$

An American Institute of Mathematics workshop showed that we are already familiar with this Hopf algebra. 

\begin{theorem}[\cite{AIM}]
The Hopf algebra $\SC(q)$ is isomorphic  to the Hopf algebra of symmetric functions in non-commuting variables $\NCSym(X)$.
\end{theorem} 

\begin{remark}
The paper \cite{AIM} actually only addresses the case when $q=2$ since that paper was using a finer supercharacter theory, but the proof for arbitrary $q$ in our current supercharacter theory follows by the same argument.  In fact, if we work purely combinatorially and ignore the representation theory, then $\SC(q)$ makes sense for arbitrary $q$.  In particular, we also get an interesting isomorphism in the case when $q=1$.
\end{remark}

Given an infinite set $X=\{X_1,X_2,\ldots\}$ of noncommuting variables, the algebra $\NCSym(X)$ has a distinguished basis of monomial symmetric functions given by
$$\{m_\mu=\sum_{(i_1,i_2,\ldots, i_n)\in O_\mu} X_{i_1}X_{i_2}\cdots X_{i_n}\mid \mu\in \cS_n,n\in \ZZ_{\geq 0}\},$$
where 
$$O_{\mu}=\{(i_1,\ldots,i_n)\in \ZZ_{\geq 1}^n\mid i_k=i_l\text{ if and only if $k$ and $l$ are in the same part of $\mu$}\},$$
and the \emph{parts} of $\mu$ are given by the transitive closure of the relation $i\sim j$ if $i\larc{}j\in \mu$ or $j\larc{}i\in \mu$. 

We will be interested in a second natural basis of $\NCSym(X)$, which is a slight variation on what is usual in the literature \cite{BZ,RS}.  Consider the power-sum symmetric functions,
$$p_\nu=\sum_{\mu\supseteq \nu} m_\mu.$$
The usual definition of $p_\nu$ uses the refinement order on set partitions rather than the subset relation in our definition.  There are several consequences from the fact that we have a different order:
\begin{enumerate}
\item[(a)]  The sums of monomial symmetric functions have fewer terms,
\item[(b)]  If we consider the function $\mathrm{NCSym(X)}\rightarrow \mathrm{Sym}(X)$ induced by allowing the variables to commute, not all the $p_\nu$ get sent to the corresponding power-sum symmetric functions (as the usual ones do).  However, if $\nu$ satisfies $i\larc{}j\in \nu$ implies $j-i=1$, then $p_\nu$ will be sent to the appropriate symmetric function.  That is, in the usual construction the image of $p_\nu$ depends on the sequence of part sizes, and in ours each $p_\nu$ gets sent to something different.
\end{enumerate}

The isomorphism 
$$\mathrm{ch}:\mathrm{SC}(q)\longrightarrow \mathrm{NCSym(X)}$$
 given in \cite{AIM} sends $\kappa_\mu$ to $m_\mu$, but there is no representation theoretic interpretation for the power-sum symmetric functions.  This paper finds a representation theoretic approach by tweaking the definition of the power-sum symmetric functions.

%%%%%%%%%%%%%%%%%%%%%%%%%%%%%%%%%%%%%%%%%%%%%%%%%%%%%%%%%%%%%%%%%%
\section{Transition matrices} \label{sec:qpowersums}

This section defines a $q$-analogue of the power-sum symmetric functions, and studies its transition matrices to the superclass function basis and the supercharacter basis.

%%%%%%%%%%%%%%%%%%%%%%%%%%%%%%%%%%%%%%%%%%%%%%%%%%%%%%%%%%%%%%%%%%
\subsection{$q$-deformations of power-sums in $\SC(q)$}

For $\nu\in \cS_n$, define 
$$\rho_\nu(q)=\sum_{\mu\supseteq \nu} \frac{1}{q^{\nst_{\mu-\nu}^\nu}} \kappa_\mu,$$
so that formally $\ch(\rho_\nu(1))=p_\nu$. 

\begin{example} The transition matrix from the $\rho$-basis to the $\kappa$-basis is given by
$$
\begin{array}{c|ccccccccccccccc}
 & \begin{tikzpicture}[scale=.2]
	\foreach \x in {0,1,2,3}
		\node (\x) at (\x,0) [inner sep =-2pt] {$\scs\bullet$};
\end{tikzpicture} 
& 
\begin{tikzpicture}[scale=.2]
	\foreach \x in {0,1,2,3}
		\node (\x) at (\x,0) [inner sep =-2pt] {$\scs\bullet$};
	\draw (0) .. controls (.25,.75) and (.75,.75) .. (1);
\end{tikzpicture} &
\begin{tikzpicture}[scale=.2]
	\foreach \x in {0,1,2,3}
		\node (\x) at (\x,0) [inner sep =-2pt] {$\scs\bullet$};
	\draw (1) .. controls (1.25,.75) and (1.75,.75) .. (2);
\end{tikzpicture} & 
\begin{tikzpicture}[scale=.2]
	\foreach \x in {0,1,2,3}
		\node (\x) at (\x,0) [inner sep =-2pt] {$\scs\bullet$};
	\draw (2) .. controls (2.25,.75) and (2.75,.75) .. (3);
\end{tikzpicture} 
& 
\begin{tikzpicture}[scale=.2]
	\foreach \x in {0,1,2,3}
		\node (\x) at (\x,0) [inner sep =-2pt] {$\scs\bullet$};
	\draw (0) .. controls (0.25,.75) and (.75,.75) .. (1);
	\draw (1) .. controls (1.25,.75) and (1.75,.75) .. (2);
\end{tikzpicture} 
& 
\begin{tikzpicture}[scale=.2]
	\foreach \x in {0,1,2,3}
		\node (\x) at (\x,0) [inner sep =-2pt] {$\scs\bullet$};
	\draw (0) .. controls (0.25,.75) and (.75,.75) .. (1);
	\draw (2) .. controls (2.25,.75) and (2.75,.75) .. (3);
\end{tikzpicture} 
& 
\begin{tikzpicture}[scale=.2]
	\foreach \x in {0,1,2,3}
		\node (\x) at (\x,0) [inner sep =-2pt] {$\scs\bullet$};
	\draw (1) .. controls (1.25,.75) and (1.75,.75) .. (2);
	\draw (2) .. controls (2.25,.75) and (2.75,.75) .. (3);
\end{tikzpicture} 
& 
\begin{tikzpicture}[scale=.2]
	\foreach \x in {0,1,2,3}
		\node (\x) at (\x,0) [inner sep =-2pt] {$\scs\bullet$};
	\draw (0) .. controls (0.25,.75) and (.75,.75) .. (1);
	\draw (1) .. controls (1.25,.75) and (1.75,.75) .. (2);
	\draw (2) .. controls (2.25,.75) and (2.75,.75) .. (3);
\end{tikzpicture} 
&
\begin{tikzpicture}[scale=.2]
	\foreach \x in {0,1,2,3}
		\node (\x) at (\x,0) [inner sep =-2pt] {$\scs\bullet$};
	\draw (0) .. controls (.5,1.25) and (1.5,1.25) .. (2);
\end{tikzpicture} 
&
\begin{tikzpicture}[scale=.2]
	\foreach \x in {0,1,2,3}
		\node (\x) at (\x,0) [inner sep =-2pt] {$\scs\bullet$};
	\draw (1) .. controls (1.5,1.25) and (2.5,1.25) .. (3);
\end{tikzpicture} 
&
\begin{tikzpicture}[scale=.2]
	\foreach \x in {0,1,2,3}
		\node (\x) at (\x,0) [inner sep =-2pt] {$\scs\bullet$};
	\draw (0) .. controls (0.25,.75) and (.75,.75) .. (1);
	\draw (1) .. controls (1.5,1.25) and (2.5,1.25) .. (3);
\end{tikzpicture} 
&
\begin{tikzpicture}[scale=.2]
	\foreach \x in {0,1,2,3}
		\node (\x) at (\x,0) [inner sep =-2pt] {$\scs\bullet$};
	\draw (0) .. controls (.5,1.25) and (1.5,1.25) .. (2);
	\draw (2) .. controls (2.25,.75) and (2.75,.75) .. (3);
\end{tikzpicture} 
&
\begin{tikzpicture}[scale=.2]
	\foreach \x in {0,1,2,3}
		\node (\x) at (\x,0) [inner sep =-2pt] {$\scs\bullet$};
	\draw (0) .. controls (.5,1.25) and (1.5,1.25) .. (2);
	\draw (1) .. controls (1.5,1.25) and (2.5,1.25) .. (3);
\end{tikzpicture} 
&
\begin{tikzpicture}[scale=.2]
	\foreach \x in {0,1,2,3}
		\node (\x) at (\x,0) [inner sep =-2pt] {$\scs\bullet$};
	\draw (0) .. controls (.75,1.75) and (2.25,1.75) .. (3);
\end{tikzpicture} 
&
\begin{tikzpicture}[scale=.2]
	\foreach \x in {0,1,2,3}
		\node (\x) at (\x,0) [inner sep =-2pt] {$\scs\bullet$};
	\draw (0) .. controls (.75,1.75) and (2.25,1.75) .. (3);
	\draw (1) .. controls (1.25,.75) and (1.75,.75) .. (2);
\end{tikzpicture} \\ \hline
  \begin{tikzpicture}[scale=.2]
	\foreach \x in {0,1,2,3}
		\node (\x) at (\x,0) [inner sep =-2pt] {$\scs\bullet$};
\end{tikzpicture} 
& 1 & 1 & 1 & 1 & 1 & 1 & 1 & 1 & 1 & 1 & 1 & 1 & 1 & 1 & 1 \\
\begin{tikzpicture}[scale=.2]
	\foreach \x in {0,1,2,3}
		\node (\x) at (\x,0) [inner sep =-2pt] {$\scs\bullet$};
	\draw (0) .. controls (.25,.75) and (.75,.75) .. (1);
\end{tikzpicture} &
0 & 1 & 0 & 0 & 1 & 1 & 0 & 1 & 0 & 0 & 1 & 0 &  0 & 0 & 0 \\
\begin{tikzpicture}[scale=.2]
	\foreach \x in {0,1,2,3}
		\node (\x) at (\x,0) [inner sep =-2pt] {$\scs\bullet$};
	\draw (1) .. controls (1.25,.75) and (1.75,.75) .. (2);
\end{tikzpicture} &  0 & 0 & 1 & 0 & 1 & 0 & 1 & 1 & 0 & 0  & 0 & 0 & 0 & 0 & 1\\
\begin{tikzpicture}[scale=.2]
	\foreach \x in {0,1,2,3}
		\node (\x) at (\x,0) [inner sep =-2pt] {$\scs\bullet$};
	\draw (2) .. controls (2.25,.75) and (2.75,.75) .. (3);
\end{tikzpicture} 
&  0 & 0 & 0  & 1 & 0 & 1 & 1 & 1 & 0 &  0 & 0 & 1 & 0 & 0 & 0\\
\begin{tikzpicture}[scale=.2]
	\foreach \x in {0,1,2,3}
		\node (\x) at (\x,0) [inner sep =-2pt] {$\scs\bullet$};
	\draw (0) .. controls (0.25,.75) and (.75,.75) .. (1);
	\draw (1) .. controls (1.25,.75) and (1.75,.75) .. (2);
\end{tikzpicture} 
& 0 &  0 & 0 & 0  & 1 & 0 & 0 & 1 & 0 & 0 &  0 & 0 & 0 & 0 & 0 \\ 
\begin{tikzpicture}[scale=.2]
	\foreach \x in {0,1,2,3}
		\node (\x) at (\x,0) [inner sep =-2pt] {$\scs\bullet$};
	\draw (0) .. controls (0.25,.75) and (.75,.75) .. (1);
	\draw (2) .. controls (2.25,.75) and (2.75,.75) .. (3);
\end{tikzpicture} 
& 0 & 0 &  0 & 0 & 0  & 1 & 0 & 1 & 0 & 0 &  0 & 0 & 0 & 0 & 0   \\
\begin{tikzpicture}[scale=.2]
	\foreach \x in {0,1,2,3}
		\node (\x) at (\x,0) [inner sep =-2pt] {$\scs\bullet$};
	\draw (1) .. controls (1.25,.75) and (1.75,.75) .. (2);
	\draw (2) .. controls (2.25,.75) and (2.75,.75) .. (3);
\end{tikzpicture} 
& 0 & 0 & 0 &  0 & 0 & 0 & 1 & 1  & 0 & 0 &  0 & 0 & 0 & 0 & 0 \\
\begin{tikzpicture}[scale=.2]
	\foreach \x in {0,1,2,3}
		\node (\x) at (\x,0) [inner sep =-2pt] {$\scs\bullet$};
	\draw (0) .. controls (0.25,.75) and (.75,.75) .. (1);
	\draw (1) .. controls (1.25,.75) and (1.75,.75) .. (2);
	\draw (2) .. controls (2.25,.75) and (2.75,.75) .. (3);
\end{tikzpicture} 
& 0 & 0 & 0 & 0 &  0 & 0 & 0 & 1   & 0 & 0 &  0 & 0 & 0 & 0 & 0 \\
\begin{tikzpicture}[scale=.2]
	\foreach \x in {0,1,2,3}
		\node (\x) at (\x,0) [inner sep =-2pt] {$\scs\bullet$};
	\draw (0) .. controls (.5,1.25) and (1.5,1.25) .. (2);
\end{tikzpicture} 
& 0 & 0 & 0 & 0 & 0 &  0 & 0 & 0 & 1    & 0 &  0 & 1 & 1 & 0 & 0\\
\begin{tikzpicture}[scale=.2]
	\foreach \x in {0,1,2,3}
		\node (\x) at (\x,0) [inner sep =-2pt] {$\scs\bullet$};
	\draw (1) .. controls (1.5,1.25) and (2.5,1.25) .. (3);
\end{tikzpicture} 
& 0 & 0 & 0 & 0 & 0 & 0 &  0 & 0 & 0 & 1    & 1 &   0 & 1 & 0 & 0\\
\begin{tikzpicture}[scale=.2]
	\foreach \x in {0,1,2,3}
		\node (\x) at (\x,0) [inner sep =-2pt] {$\scs\bullet$};
	\draw (0) .. controls (0.25,.75) and (.75,.75) .. (1);
	\draw (1) .. controls (1.5,1.25) and (2.5,1.25) .. (3);
\end{tikzpicture} 
& 0 & 0 & 0 & 0 & 0 & 0 & 0 &  0 & 0 & 0 & 1    &   0 & 0 & 0 & 0\\
\begin{tikzpicture}[scale=.2]
	\foreach \x in {0,1,2,3}
		\node (\x) at (\x,0) [inner sep =-2pt] {$\scs\bullet$};
	\draw (0) .. controls (.5,1.25) and (1.5,1.25) .. (2);
	\draw (2) .. controls (2.25,.75) and (2.75,.75) .. (3);
\end{tikzpicture} 
& 0 & 0 & 0 & 0 & 0 & 0 & 0 & 0 &  0 & 0 & 0 & 1  & 0 & 0 & 0\\
\begin{tikzpicture}[scale=.2]
	\foreach \x in {0,1,2,3}
		\node (\x) at (\x,0) [inner sep =-2pt] {$\scs\bullet$};
	\draw (0) .. controls (.5,1.25) and (1.5,1.25) .. (2);
	\draw (1) .. controls (1.5,1.25) and (2.5,1.25) .. (3);
\end{tikzpicture} 
& 0 & 0 & 0 & 0 & 0 & 0 & 0 & 0 & 0 &  0 & 0 & 0 & 1  & 0 & 0\\
\begin{tikzpicture}[scale=.2]
	\foreach \x in {0,1,2,3}
		\node (\x) at (\x,0) [inner sep =-2pt] {$\scs\bullet$};
	\draw (0) .. controls (.75,1.75) and (2.25,1.75) .. (3);
\end{tikzpicture} 
& 0 & 0 & 0 & 0 & 0 & 0 & 0 & 0 & 0 & 0 &  0 & 0 & 0 & 1  & q^{-1}\\
\begin{tikzpicture}[scale=.2]
	\foreach \x in {0,1,2,3}
		\node (\x) at (\x,0) [inner sep =-2pt] {$\scs\bullet$};
	\draw (0) .. controls (.75,1.75) and (2.25,1.75) .. (3);
	\draw (1) .. controls (1.25,.75) and (1.75,.75) .. (2);
\end{tikzpicture}  & 0 & 0 & 0 & 0 & 0 & 0 & 0 & 0 & 0 & 0 & 0 &  0 & 0 & 0 & 1 
\end{array}
$$
\end{example}

The following proposition computes the inverse of this matrix.

\begin{proposition}  \label{kappatorhoinverse}
$$\kappa_\mu=\sum_{\nu\supseteq \mu} \frac{(-1)^{|\nu-\mu|}}{q^{\nst_{\nu-\mu}^{\nu}}} \rho_\nu(q).$$
\end{proposition}

\begin{proof}
We wish to show that 
\begin{align*}
\rho_\nu(q)&=\sum_{\mu\supseteq \nu}\frac{1}{q^{\nst_{\mu-\nu}^\nu}} \sum_{\lambda\supseteq \mu}\frac{(-1)^{|\lambda-\mu|}}{q^{\nst_{\lambda-\mu}^{\lambda}}} \rho_\lambda(q)\\
&=\sum_{\lambda\supseteq\mu\supseteq \nu}\frac{(-1)^{|\lambda-\mu|}}{q^{\nst_{\mu-\nu}^\nu}q^{\nst_{\lambda-\mu}^{\lambda}}}   \rho_\lambda(q).
\end{align*}
In other words,
$$\sum_{\lambda\supseteq\mu\supseteq \nu}\frac{(-1)^{|\lambda-\mu|}}{q^{\nst_{\mu-\nu}^\nu}q^{\nst_{\lambda-\mu}^{\lambda}}}=\left\{\begin{array}{ll} 1 & \text{if $\lambda=\nu$,}\\ 0 &\text{otherwise.}\end{array}\right.$$
If $\lambda=\nu$, then the sum has one term which is
$$\frac{1}{q^{\nst^\nu_\emptyset}q^{\nst^\nu_{\emptyset}}}=1.$$

Assume $\lambda\neq \nu$.  To establish 
$$ \sum_{\lambda\supseteq\mu\supseteq \nu}\frac{(-1)^{|\lambda-\mu|}}{q^{\nst_{\mu-\nu}^\nu}q^{\nst_{\lambda-\mu}^{\lambda}}}  =0$$
 we define an involution $\iota$ on the set $\{\nu\subseteq \mu\subseteq \lambda\}$ such that 
\begin{enumerate}
\item[(a)] $(-1)^{|\lambda-\iota(\mu)|}=-(-1)^{|\lambda-\mu|}$,
\item[(b)] $q^{\nst_{\iota(\mu)-\nu}^\nu} q^{\nst_{\lambda-\iota(\mu)}^{\lambda}}=q^{\nst_{\mu-\nu}^\nu}  q^{\nst_{\lambda-\mu}^{\lambda}}$.
\end{enumerate}
Let $\alpha=i\larc{}l\in \lambda-\nu$ be maximal with respect to the statistic $l-i$ (the particular choice is irrelevant).
Define the involution by
$$\iota(\mu)=\left\{\begin{array}{ll} \mu\cup\{\alpha\} & \text{if $\alpha\notin \mu$},\\ \mu-\{\alpha\} & \text{if $\alpha\in \mu$}.\end{array}\right.$$
Clearly (a) holds under this involution.  For (b), suppose $\alpha\in \mu$.  then 
\begin{align*}
q^{\nst_{\iota(\mu)-\nu}^{\nu}} & = q^{\nst_{\mu-\nu}^\nu+\#\{i'<i<l<l'\mid i'\slarc{} l'\in \nu\}}\\
&= q^{\nst_{\mu-\nu}^\nu+\#\{i'<i<l<l'\mid i'\slarc{} l'\in \lambda\}}
\end{align*}
by the maximality in the choice of $\alpha$.  On the other hand,
$$
q^{\nst_{\lambda-\iota(\mu)}^{\lambda}}= q^{\nst_{\lambda-\mu}^{\lambda}-\#\{i'<i<l<l'\mid i'\slarc{} l'\in \lambda\}}.
$$
Condition (b)  follows.
\end{proof}

Let $\chi^\lambda_\mu$ denote the value of the supercharacter $\chi^\lambda$ on the superclass indexed by $\mu$.  By Proposition \ref{kappatorhoinverse},
\begin{align*}
\chi^\lambda & = \sum_{\mu} \chi^\lambda_\mu \kappa_\mu\\
& = \sum_{\mu} \chi^\lambda_\mu \sum_{\nu\supseteq \mu} \frac{(-1)^{|\nu-\mu|}}{q^{\nst_{\nu-\mu}^\nu}} \rho_\nu(q)\\
& =\sum_{\nu} \bigg(\sum_{\mu\subseteq \nu} \chi^\lambda_\mu  \frac{(-1)^{|\nu-\mu|}}{q^{\nst_{\nu-\mu}^\nu}}\bigg) \rho_\nu(q).
\end{align*}
We are interested in these coefficients of the $\rho_\nu(q)$.

For $\lambda, \mu\in \cS_n$, let
\begin{align*}
\mathrm{cflt}(\mu)&=\{j\larc{}k\mid \text{there exists $i\larc{}l\in \mu$ with $i=j<k<l$ or $i<j<k=l$}\}\\
\mathrm{snst}^\lambda_\mu&=\#\{i<j<k<l\mid i\larc{}l\in \lambda, j\larc{}k\in \mu-\mathrm{cflt}(\lambda)\}
\end{align*}
be the sets of \emph{arcs conflicting with $\mu$} and the set of \emph{strictly nested pairs}, respectively.

\begin{theorem} \label{CoefficientFormula} For $\lambda,\nu\in \cS_n$,
\begin{align*}
\sum_{\mu\subseteq \nu} & \chi^\lambda_\mu  \frac{(-1)^{|\nu-\mu|}}{q^{\nst_{\nu-\mu}^\nu}}\\
 &= \frac{(-1)^{|\nu|}q^{\dim(\lambda)}(q-1)^{|\lambda-\nu|}}{q^{|\lambda|+\mathrm{snst}_\nu^\lambda+\mathrm{nst}_\nu^\nu}}\bigg(\prod_{i\slarc{}j\in\nu\cap \lambda}\hspace{-.2cm} (q-1)q^{\nst_{i\slarc{}j}^\lambda}+q^{\nst_{i\slarc{}j}^\nu}\bigg) \bigg(\prod_{i\slarc{}j\in\nu- \lambda\atop i\slarc{}j\notin \mathrm{cflt}(\lambda)} \hspace{-.3cm} q^{\nst_{i\slarc{}j}^\lambda}-q^{\nst_{i\slarc{}j}^\nu}\bigg).
\end{align*}
\end{theorem}
\begin{proof}  For the purpose of this proof, let $t=q-1$.
First note that by (\ref{SupercharacterFormula}) we have that
\begin{align*}
\chi^\lambda_\mu=q^{\dim(\lambda)-|\lambda|}\bigg(\prod_{i\slarc{}j\in \lambda-\mu}t\bigg)\bigg(\prod_{i\slarc{}j\in \lambda\cap\mu} \frac{-1}{q^{\nst_{i\slarc{}j}^\lambda}}\bigg)\bigg(\prod_{i\slarc{}j\in \mu-\lambda\atop i\slarc{}j\notin \mathrm{cflt}(\lambda)} \frac{1}{q^{\nst_{i\slarc{}j}^\lambda}}\bigg)\bigg(\prod_{i\slarc{}j\in \mu-\lambda\atop i\slarc{}j\in \mathrm{cflt}(\lambda)} 0\bigg),
\end{align*}
Plug into the coefficient formula to get
\begin{align*}
&\sum_{\mu\subseteq \nu}  \chi^\lambda_\mu  \frac{(-1)^{|\nu-\mu|}}{q^{\nst_{\nu-\mu}^\nu}} \\
&=\frac{q^{\dim(\lambda)}}{q^{|\lambda|}} \sum_{\mu\subseteq \nu}\bigg(\prod_{i\slarc{}j\in \lambda-\mu}\hspace{-.15cm}t\bigg)\bigg(\prod_{i\slarc{}j\in \lambda\cap\mu} \frac{-1}{q^{\nst_{i\slarc{}j}^\lambda}}\bigg)\bigg(\prod_{i\slarc{}j\in \mu-\lambda\atop i\slarc{}j\notin \mathrm{cflt}(\lambda)} \frac{1}{q^{\nst_{i\slarc{}j}^\lambda}}\bigg)\bigg(\prod_{i\slarc{}j\in \mu-\lambda\atop i\slarc{}j\in \mathrm{cflt}(\lambda)}\hspace{-.2cm} 0\bigg) \frac{(-1)^{|\nu-\mu|}}{q^{\nst_{\nu-\mu}^\nu}}\\
&= \frac{q^{\dim(\lambda)}}{q^{|\lambda|}} \sum_{\mu\subseteq \nu}\bigg(\prod_{i\slarc{}j\in \lambda-\mu}\hspace{-.15cm} t\bigg)\bigg(\prod_{i\slarc{}j\in \lambda\cap\mu} \frac{-1}{q^{\nst_{i\slarc{}j}^\lambda}}\bigg)\bigg(\prod_{i\slarc{}j\in \mu-\lambda\atop i\slarc{}j\notin \mathrm{cflt}(\lambda)} \frac{1}{q^{\nst_{i\slarc{}j}^\lambda}}\bigg)\bigg(\prod_{i\slarc{}j\in \mu-\lambda\atop i\slarc{}j\in \mathrm{cflt}(\lambda)} \hspace{-.2cm} 0\bigg) \bigg(\prod_{i\slarc{}j\in \nu-\mu} \frac{-1}{q^{\nst_{i\slarc{}j}^\nu}}\bigg)\\
&= \frac{q^{\dim(\lambda)}t^{|\lambda-\nu|}}{q^{|\lambda|}} \\
&\hspace*{.75cm}\cdot\sum_{\mu\subseteq \nu}\bigg(\prod_{i\slarc{}j\in (\lambda\cap\nu)-\mu}\hspace{-.15cm} t\bigg)\bigg(\prod_{i\slarc{}j\in \lambda\cap\mu} \frac{-1}{q^{\nst_{i\slarc{}j}^\lambda}}\bigg)\bigg(\prod_{i\slarc{}j\in \mu-\lambda\atop i\slarc{}j\notin \mathrm{cflt}(\lambda)} \frac{1}{q^{\nst_{i\slarc{}j}^\lambda}}\bigg)\bigg(\prod_{i\slarc{}j\in \mu-\lambda\atop i\slarc{}j\in \mathrm{cflt}(\lambda)} \hspace{-.2cm} 0\bigg) \bigg(\prod_{i\slarc{}j\in \nu-\mu} \frac{-1}{q^{\nst_{i\slarc{}j}^\nu}}\bigg)
\end{align*}
Thus,
$$\sum_{\mu\subseteq \nu}  \chi^\lambda_\mu  \frac{(-1)^{|\nu-\mu|}}{q^{\nst_{\nu-\mu}^\nu}} = \frac{q^{\dim(\lambda)}t^{|\lambda-\nu|}}{q^{|\lambda|}} \sum_{\mu\subseteq \nu}  \prod_{i\slarc{}j\in \nu}\mathrm{val}_\mu^\lambda(i\larc{}j),$$
where
\begin{equation}\label{ValueFunction}
\mathrm{val}_\mu^\lambda(i\larc{}j)=\left\{\begin{array}{ll}
-q^{-\nst^\lambda_{i\slarc{}j}} & \text{if $i\larc{}j\in \lambda\cap\mu$,}\\
q^{-\nst_{i\slarc{} j}^\mu}  & \text{if $i\larc{}j\in \mu-\lambda$, $i\larc{}j\notin \mathrm{cflt}(\lambda)$,}\\
0 & \text{if $i\larc{}j\in \mu-\lambda$, $i\larc{}j\in \mathrm{cflt}(\lambda)$,}\\
-tq^{-\nst_{i\slarc{}j}^\nu} & \text{if $i\slarc{}j\in \lambda-\mu$,}\\
-q^{-\nst_{i\slarc{}j}^\nu} & \text{if $i\slarc{}j\notin \lambda\cup\mu$,}\\

\end{array}\right.
\end{equation}
Fix $k\larc{}l\in \nu$.  Then
\begin{align*}
\sum_{\mu\subseteq \nu} & \chi^\lambda_\mu  \frac{(-1)^{|\nu-\mu|}}{q^{\nst_{\nu-\mu}^\nu}} = q^{\dim(\lambda)-|\lambda|} \bigg(\sum_{\mu\subseteq \nu\atop k\slarc{}l\in \mu}  \prod_{i\slarc{}j\in \nu}\mathrm{val}_\mu^\lambda(i\larc{}j)+\sum_{\mu\subseteq \nu\atop k\slarc{}l\notin \mu}  \prod_{i\slarc{}j\in \nu}\mathrm{val}_\mu^\lambda(i\larc{}j)\bigg)\\ 
&=q^{\dim(\lambda)-|\lambda|} \bigg(\sum_{\mu\subseteq \nu\atop k\slarc{}l\in \mu} \mathrm{val}_\mu^\lambda(k\larc{}l) \prod_{i\slarc{}j\in \nu\atop i\slarc{}j\neq k\slarc{} l}\mathrm{val}_\mu^\lambda(i\larc{}j) + \sum_{\mu\subseteq \nu\atop k\slarc{}l\in \mu} \mathrm{val}_\mu^\lambda(k\larc{}l)  \prod_{i\slarc{}j\in \nu\atop i\slarc{}j\neq k\slarc{} l}\mathrm{val}_\mu^\lambda(i\larc{}j)\bigg)\\
&=q^{\dim(\lambda)-|\lambda|} \bigg(\sum_{\mu\subseteq \nu\atop k\slarc{}l\in \mu} \mathrm{val}_\nu^\lambda(k\larc{}l) \prod_{i\slarc{}j\in \nu\atop i\slarc{}j\neq k\slarc{} l}\mathrm{val}_\mu^\lambda(i\larc{}j) + \sum_{\mu\subseteq \nu\atop k\slarc{}l\in \mu} \mathrm{val}_\emptyset^\lambda(k\larc{}l)  \prod_{i\slarc{}j\in \nu\atop i\slarc{}j\neq k\slarc{} l}\mathrm{val}_\mu^\lambda(i\larc{}j)\bigg).
\end{align*}
Note that
\begin{align*}
\sum_{\mu\subseteq \nu\atop k\slarc{}l\in \mu} \prod_{i\slarc{}j\in \nu\atop i\slarc{}j\neq k\slarc{} l}\mathrm{val}_\mu^\lambda(i\larc{}j) &= \sum_{\mu\subseteq \nu\atop k\slarc{}l\in \mu}   \prod_{i\slarc{}j\in \nu\atop i\slarc{}j\neq k\slarc{} l}\mathrm{val}_\mu^\lambda(i\larc{}j),\\
&=\sum_{\mu\subseteq \nu-\{k\slarc{}l\}}    \prod_{i\slarc{}j\in \nu-\{k\slarc{}l\}}\mathrm{val}_\mu^\lambda(i\larc{}j),
\end{align*}
Thus,
\begin{align*}
\sum_{\mu\subseteq \nu}  \chi^\lambda_\mu  \frac{(-1)^{|\nu-\mu|}}{q^{\nst_{\nu-\mu}^\nu}} & =  q^{\dim(\lambda)-|\lambda|} ( \mathrm{val}_\nu^\lambda(k\larc{}l)+ \mathrm{val}_\emptyset^\lambda(k\larc{}l))\bigg(\sum_{\mu\subseteq \nu-\{k\slarc{}l\}}    \prod_{i\slarc{}j\in \nu-\{k\slarc{}l\}}\mathrm{val}_\mu^\lambda(i\larc{}j)\bigg)\\
&=q^{\dim(\lambda)-|\lambda|} \prod_{i\slarc{}j\in \nu} ( \mathrm{val}_\nu^\lambda(i\larc{}j)+ \mathrm{val}_\emptyset^\lambda(i\larc{}j)),
\end{align*}
where the second equality is obtained by iterating (ie. fix $k'\larc{}l'\in \nu-\{k\larc{}l\}$, etc.).  

By separating into the cases given by (\ref{ValueFunction}), we obtain
\begin{align*}
\sum_{\mu\subseteq \nu} & \chi^\lambda_\mu  \frac{(-1)^{|\nu-\mu|}}{q^{\nst_{\nu-\mu}^\nu}}\\
&= \frac{q^{\dim(\lambda)}t^{|\lambda-\nu|}}{q^{|\lambda|}} \bigg(\prod_{i\slarc{}j\in \nu\cap \lambda} \frac{-1}{q^{\nst^\lambda_{i\slarc{}j}}}+\frac{-t}{q^{\nst_{i\slarc{}j}^\nu}}\bigg)\bigg(\prod_{i\slarc{}j\in \nu-\lambda\atop i\slarc{}j\notin \mathrm{cflt}(\lambda)} \frac{1}{q^{\nst^\lambda_{i\slarc{}j}}}+\frac{-1}{q^{\nst_{i\slarc{}j}^\nu}}\bigg)\bigg(\prod_{i\slarc{}j\in \nu-\lambda\atop i\slarc{}j\in \mathrm{cflt}(\lambda)}\frac{-1}{q^{\nst_{i\slarc{}j}^\nu}}\bigg)\\
&=\frac{(-1)^{|\nu|}q^{\dim(\lambda)}t^{|\lambda-\nu|}}{q^{|\lambda|+\mathrm{snst}_\nu^\lambda+\nst_\nu^\nu}} 
\bigg(\prod_{i\slarc{}j\in \nu\cap \lambda} tq^{\nst^\lambda_{i\slarc{}j}}+q^{\nst_{i\slarc{}j}^\nu}\bigg)\bigg(\prod_{i\slarc{}j\in \nu-\lambda\atop i\slarc{}j\notin \mathrm{cflt}(\lambda)}q^{\nst^\lambda_{i\slarc{}j}}-q^{\nst_{i\slarc{}j}^\nu}\bigg),
\end{align*}
as desired.
\end{proof}

\begin{example}
The transition matrix from the $\chi$-basis to the $\rho$-basis is therefore 
 $$
\begin{array}{c|cc@{}c@{}cccccccccccc}
 & \begin{tikzpicture}[scale=.2]
	\foreach \x in {0,1,2,3}
		\node (\x) at (\x,0) [inner sep =-2pt] {$\scs\bullet$};
\end{tikzpicture} 
& 
\begin{tikzpicture}[scale=.2]
	\foreach \x in {0,1,2,3}
		\node (\x) at (\x,0) [inner sep =-2pt] {$\scs\bullet$};
	\draw (0) .. controls (.25,.75) and (.75,.75) .. (1);
\end{tikzpicture} &
\begin{tikzpicture}[scale=.2]
	\foreach \x in {0,1,2,3}
		\node (\x) at (\x,0) [inner sep =-2pt] {$\scs\bullet$};
	\draw (1) .. controls (1.25,.75) and (1.75,.75) .. (2);
\end{tikzpicture} & 
\begin{tikzpicture}[scale=.2]
	\foreach \x in {0,1,2,3}
		\node (\x) at (\x,0) [inner sep =-2pt] {$\scs\bullet$};
	\draw (2) .. controls (2.25,.75) and (2.75,.75) .. (3);
\end{tikzpicture} 
& 
\begin{tikzpicture}[scale=.2]
	\foreach \x in {0,1,2,3}
		\node (\x) at (\x,0) [inner sep =-2pt] {$\scs\bullet$};
	\draw (0) .. controls (0.25,.75) and (.75,.75) .. (1);
	\draw (1) .. controls (1.25,.75) and (1.75,.75) .. (2);
\end{tikzpicture} 
& 
\begin{tikzpicture}[scale=.2]
	\foreach \x in {0,1,2,3}
		\node (\x) at (\x,0) [inner sep =-2pt] {$\scs\bullet$};
	\draw (0) .. controls (0.25,.75) and (.75,.75) .. (1);
	\draw (2) .. controls (2.25,.75) and (2.75,.75) .. (3);
\end{tikzpicture} 
& 
\begin{tikzpicture}[scale=.2]
	\foreach \x in {0,1,2,3}
		\node (\x) at (\x,0) [inner sep =-2pt] {$\scs\bullet$};
	\draw (1) .. controls (1.25,.75) and (1.75,.75) .. (2);
	\draw (2) .. controls (2.25,.75) and (2.75,.75) .. (3);
\end{tikzpicture} 
& 
\begin{tikzpicture}[scale=.2]
	\foreach \x in {0,1,2,3}
		\node (\x) at (\x,0) [inner sep =-2pt] {$\scs\bullet$};
	\draw (0) .. controls (0.25,.75) and (.75,.75) .. (1);
	\draw (1) .. controls (1.25,.75) and (1.75,.75) .. (2);
	\draw (2) .. controls (2.25,.75) and (2.75,.75) .. (3);
\end{tikzpicture} 
&
\begin{tikzpicture}[scale=.2]
	\foreach \x in {0,1,2,3}
		\node (\x) at (\x,0) [inner sep =-2pt] {$\scs\bullet$};
	\draw (0) .. controls (.5,1.25) and (1.5,1.25) .. (2);
\end{tikzpicture} 
&
\begin{tikzpicture}[scale=.2]
	\foreach \x in {0,1,2,3}
		\node (\x) at (\x,0) [inner sep =-2pt] {$\scs\bullet$};
	\draw (1) .. controls (1.5,1.25) and (2.5,1.25) .. (3);
\end{tikzpicture} 
&
\begin{tikzpicture}[scale=.2]
	\foreach \x in {0,1,2,3}
		\node (\x) at (\x,0) [inner sep =-2pt] {$\scs\bullet$};
	\draw (0) .. controls (0.25,.75) and (.75,.75) .. (1);
	\draw (1) .. controls (1.5,1.25) and (2.5,1.25) .. (3);
\end{tikzpicture} 
&
\begin{tikzpicture}[scale=.2]
	\foreach \x in {0,1,2,3}
		\node (\x) at (\x,0) [inner sep =-2pt] {$\scs\bullet$};
	\draw (0) .. controls (.5,1.25) and (1.5,1.25) .. (2);
	\draw (2) .. controls (2.25,.75) and (2.75,.75) .. (3);
\end{tikzpicture} 
&
\begin{tikzpicture}[scale=.2]
	\foreach \x in {0,1,2,3}
		\node (\x) at (\x,0) [inner sep =-2pt] {$\scs\bullet$};
	\draw (0) .. controls (.5,1.25) and (1.5,1.25) .. (2);
	\draw (1) .. controls (1.5,1.25) and (2.5,1.25) .. (3);
\end{tikzpicture} 
&
\begin{tikzpicture}[scale=.2]
	\foreach \x in {0,1,2,3}
		\node (\x) at (\x,0) [inner sep =-2pt] {$\scs\bullet$};
	\draw (0) .. controls (.75,1.75) and (2.25,1.75) .. (3);
\end{tikzpicture} 
&
\begin{tikzpicture}[scale=.2]
	\foreach \x in {0,1,2,3}
		\node (\x) at (\x,0) [inner sep =-2pt] {$\scs\bullet$};
	\draw (0) .. controls (.75,1.75) and (2.25,1.75) .. (3);
	\draw (1) .. controls (1.25,.75) and (1.75,.75) .. (2);
\end{tikzpicture} \\ \hline
  \begin{tikzpicture}[scale=.2]
	\foreach \x in {0,1,2,3}
		\node (\x) at (\x,0) [inner sep =-2pt] {$\scs\bullet$};
\end{tikzpicture} 
& 1 & 0 & 0 & 0 & 0 & 0 & 0 & 0 & 0 & 0 & 0 & 0 & 0 & 0 & 0 \\
\begin{tikzpicture}[scale=.2]
	\foreach \x in {0,1,2,3}
		\node (\x) at (\x,0) [inner sep =-2pt] {$\scs\bullet$};
	\draw (0) .. controls (.25,.75) and (.75,.75) .. (1);
\end{tikzpicture} 
& t & -q & 0 & 0 & 0 & 0 & 0 & 0 & 0 & 0 & 0 & 0 &  0 & 0 & 0 \\
\begin{tikzpicture}[scale=.2]
	\foreach \x in {0,1,2,3}
		\node (\x) at (\x,0) [inner sep =-2pt] {$\scs\bullet$};
	\draw (1) .. controls (1.25,.75) and (1.75,.75) .. (2);
\end{tikzpicture} 
&  t & 0 & -q & 0 & 0 & 0 & 0 & 0 & 0 & 0  & 0 & 0 & 0 & 0 & 0\\
\begin{tikzpicture}[scale=.2]
	\foreach \x in {0,1,2,3}
		\node (\x) at (\x,0) [inner sep =-2pt] {$\scs\bullet$};
	\draw (2) .. controls (2.25,.75) and (2.75,.75) .. (3);
\end{tikzpicture} 
&  t & 0 & 0  & -q & 0 & 0 & 0 & 0 & 0 &  0 & 0 & 0 & 0 & 0 & 0\\
\begin{tikzpicture}[scale=.2]
	\foreach \x in {0,1,2,3}
		\node (\x) at (\x,0) [inner sep =-2pt] {$\scs\bullet$};
	\draw (0) .. controls (0.25,.75) and (.75,.75) .. (1);
	\draw (1) .. controls (1.25,.75) and (1.75,.75) .. (2);
\end{tikzpicture} 
& t^2 &  -tq & -tq & 0  & q^2 & 0 & 0 & 0 & 0 & 0 &  0 & 0 & 0 & 0 & 0 \\ 
\begin{tikzpicture}[scale=.2]
	\foreach \x in {0,1,2,3}
		\node (\x) at (\x,0) [inner sep =-2pt] {$\scs\bullet$};
	\draw (0) .. controls (0.25,.75) and (.75,.75) .. (1);
	\draw (2) .. controls (2.25,.75) and (2.75,.75) .. (3);
\end{tikzpicture} 
& t^2 & -tq &  0 & -tq & 0  & q^2 & 0 & 0 & 0 & 0 &  0 & 0 & 0 & 0 & 0   \\
\begin{tikzpicture}[scale=.2]
	\foreach \x in {0,1,2,3}
		\node (\x) at (\x,0) [inner sep =-2pt] {$\scs\bullet$};
	\draw (1) .. controls (1.25,.75) and (1.75,.75) .. (2);
	\draw (2) .. controls (2.25,.75) and (2.75,.75) .. (3);
\end{tikzpicture} 
& t^2 & 0 & -tq &  -tq & 0 & 0 & q^2 & 0  & 0 & 0 &  0 & 0 & 0 & 0 & 0 \\
\begin{tikzpicture}[scale=.2]
	\foreach \x in {0,1,2,3}
		\node (\x) at (\x,0) [inner sep =-2pt] {$\scs\bullet$};
	\draw (0) .. controls (0.25,.75) and (.75,.75) .. (1);
	\draw (1) .. controls (1.25,.75) and (1.75,.75) .. (2);
	\draw (2) .. controls (2.25,.75) and (2.75,.75) .. (3);
\end{tikzpicture} 
& t^3 & -t^2q & -t^2q & -t^2q &  tq^2 & tq^2 & tq^2 & -q^3   & 0 & 0 &  0 & 0 & 0 & 0 & 0 \\
\begin{tikzpicture}[scale=.2]
	\foreach \x in {0,1,2,3}
		\node (\x) at (\x,0) [inner sep =-2pt] {$\scs\bullet$};
	\draw (0) .. controls (.5,1.25) and (1.5,1.25) .. (2);
\end{tikzpicture} 
& tq & 0 & 0 & 0 & 0 &  0 & 0 & 0 & -q^2    & 0 &  0 & 0 & 0 & 0 & 0\\
\begin{tikzpicture}[scale=.2]
	\foreach \x in {0,1,2,3}
		\node (\x) at (\x,0) [inner sep =-2pt] {$\scs\bullet$};
	\draw (1) .. controls (1.5,1.25) and (2.5,1.25) .. (3);
\end{tikzpicture} 
& tq & 0 & 0 & 0 & 0 & 0 &  0 & 0 & 0 & -q^2    & 0 &   0 & 0 & 0 & 0\\
\begin{tikzpicture}[scale=.2]
	\foreach \x in {0,1,2,3}
		\node (\x) at (\x,0) [inner sep =-2pt] {$\scs\bullet$};
	\draw (0) .. controls (0.25,.75) and (.75,.75) .. (1);
	\draw (1) .. controls (1.5,1.25) and (2.5,1.25) .. (3);
\end{tikzpicture} 
& t^2q & -tq^2 & 0 & 0 & 0 & 0 & 0 &  0 & 0 & -tq^2 & q^3    &   0 & 0 & 0 & 0\\
\begin{tikzpicture}[scale=.2]
	\foreach \x in {0,1,2,3}
		\node (\x) at (\x,0) [inner sep =-2pt] {$\scs\bullet$};
	\draw (0) .. controls (.5,1.25) and (1.5,1.25) .. (2);
	\draw (2) .. controls (2.25,.75) and (2.75,.75) .. (3);
\end{tikzpicture} 
& t^2q & 0 & 0 & -tq^2 & 0 & 0 & 0 & 0 &  -tq^2 & 0 & 0 & q^3  & 0 & 0 & 0\\
\begin{tikzpicture}[scale=.2]
	\foreach \x in {0,1,2,3}
		\node (\x) at (\x,0) [inner sep =-2pt] {$\scs\bullet$};
	\draw (0) .. controls (.5,1.25) and (1.5,1.25) .. (2);
	\draw (1) .. controls (1.5,1.25) and (2.5,1.25) .. (3);
\end{tikzpicture} 
& t^2q^2 & 0 & 0 & 0 & 0 & 0 & 0 & 0 & -tq^3 &  -tq^3 & 0 & 0 & q^4  & 0 & 0\\
\begin{tikzpicture}[scale=.2]
	\foreach \x in {0,1,2,3}
		\node (\x) at (\x,0) [inner sep =-2pt] {$\scs\bullet$};
	\draw (0) .. controls (.75,1.75) and (2.25,1.75) .. (3);
\end{tikzpicture} 
& tq^2 & 0 & -t^2q & 0 & 0 & 0 & 0 & 0 & 0 & 0 &  0 & 0 & 0 & -q^3  & 0\\
\begin{tikzpicture}[scale=.2]
	\foreach \x in {0,1,2,3}
		\node (\x) at (\x,0) [inner sep =-2pt] {$\scs\bullet$};
	\draw (0) .. controls (.75,1.75) and (2.25,1.75) .. (3);
	\draw (1) .. controls (1.25,.75) and (1.75,.75) .. (2);
\end{tikzpicture}  
& t^2q^2 & 0 & t(q^3-q^2+q) & 0 & 0 & 0 & 0 & 0 & 0 & 0 & 0 &  0 & 0 & -tq^3 & q^3 
\end{array}
$$
\end{example}

As we can see in the above example, the matrix appears to be lower-triangular.   The following defines a total order on $\cS_n$ that makes this clear, while respecting our poset of set partition inclusion.

For $\lambda ,\mu\in \cS_n$, let $\dimv(\lambda)$ be the integer partition given by the multiset $\{l-i\mid i\larc{}l\in \lambda\}$ and $\rnode(\lambda)$ be the integer partition given by the set $\{l\mid i\larc{}l\in \lambda\}$.  For example, 
$$\dimv\left(
\begin{tikzpicture}[scale=.5,baseline=0.25cm]
	\foreach \x in {0,...,5}
		\node (\x) at (\x,0) [inner sep = -1pt] {$\bullet$};
	\draw (0) .. controls (1,2) and (3,2) .. (4);
	\draw (1) .. controls (1.25,.75) and (1.75,.75) .. (2);
	\draw (2) .. controls (2.25,.75) and (2.75,.75) .. (3);
	\draw (3) .. controls (3.5,1.25) and (4.5,1.25) .. (5);
\end{tikzpicture}\right)
=(4,2,1,1)\qquad \text{and}\qquad
\rnode\left(
\begin{tikzpicture}[scale=.5,baseline=0.25cm]
	\foreach \x in {0,...,5}
		\node (\x) at (\x,0) [inner sep = -1pt] {$\bullet$};
	\draw (0) .. controls (1,2) and (3,2) .. (4);
	\draw (1) .. controls (1.25,.75) and (1.75,.75) .. (2);
	\draw (2) .. controls (2.25,.75) and (2.75,.75) .. (3);
	\draw (3) .. controls (3.5,1.25) and (4.5,1.25) .. (5);
\end{tikzpicture}\right)=(6,5,4,3).$$
Note that $\dim(\lambda)=|\dimv(\lambda)|$ or the size of the corresponding integer partition.

Define a total order $\leq$ on $\cS_n$ by
\begin{enumerate}
\item[(a)] $\lambda\geq \mu$ if $\dimv(\lambda)\geq_{\text{lex}} \dimv(\mu)$, where $\geq_{\text{lex}}$ is the biggest part to smallest part lexicographic order on integer partitions, and 
\item[(b)] If $\dimv(\lambda)= \dimv(\mu)$, then $\rnode(\lambda)\geq_{\text{lex}} \rnode(\mu)$.
\end{enumerate}
For example, for $n=4$ we have in increasing order,
$$\begin{array}{|c|c|c|c|c|c|} \hline
\lambda & 
\begin{tikzpicture}[scale=.5]
	\foreach \x in {0,1,2}
		\node (\x) at (\x,0) [inner sep =-2pt] {$\bullet$};
\end{tikzpicture}
&
\begin{tikzpicture}[scale=.5]
	\foreach \x in {0,1,2}
		\node (\x) at (\x,0)  [inner sep =-2pt]  {$\bullet$};
		\draw (0) .. controls (.25,.5) and (.75,.5) .. (1);
\end{tikzpicture}
&
\begin{tikzpicture}[scale=.5]
	\foreach \x in {0,1,2}
		\node (\x) at (\x,0)  [inner sep =-2pt]  {$\bullet$};
		\draw (1) .. controls (1.25,.5) and (1.75,.5) .. (2);
\end{tikzpicture} 
&
\begin{tikzpicture}[scale=.5]
	\foreach \x in {0,1,2}
		\node (\x) at (\x,0)  [inner sep =-2pt]  {$\bullet$};
		\draw (0) .. controls (.25,.5) and (.75,.5) .. (1);
		\draw (1) .. controls (1.25,.5) and (1.75,.5) .. (2);
\end{tikzpicture} &
\begin{tikzpicture}[scale=.5]
	\foreach \x in {0,1,2}
		\node (\x) at (\x,0)  [inner sep =-2pt]  {$\bullet$};
		\draw (0) .. controls (.5,1) and (1.5,1) .. (2);
\end{tikzpicture}\\ \hline
\dimv(\lambda) & \emptyset & (1) & (1) & (1,1) & (2) \\ \hline
\rnode(\lambda) & \emptyset & (2) & (3) & (3,2) & (3) \\ \hline
 \end{array}$$
 This is also the order used in all of our $n=4$ transition matrices above.
 
 \begin{remark}
 For our purposes any poset that respects the poset obtained by using (a) above is sufficient.  We add (b) only to get a total order.
 \end{remark}
 
 \begin{corollary}  Let $\nu,\lambda\in \cS_n$.  If $\nu>\lambda$, then 
$$\sum_{\mu\subseteq\nu} \chi^\lambda_\mu  \frac{(-1)^{|\nu-\mu|}}{q^{\nst_{\nu-\mu}^\nu}}=0.$$
\end{corollary}
\begin{proof}
Suppose there exists $j\larc{}k\in \nu-(\lambda\cup\mathrm{cflt}(\lambda))$ such that $\nst_{j\slarc{}k}^\lambda=0$.  Pick such an arc maximal with respect to $k-j$. If $\nst_{j\slarc{}k}^\nu\neq 0$, then by the maximality of $j\larc{}k$ there would exist $i\slarc{}l\in \nu\cap\mathrm{cflt}(\lambda)$ such that $i<j<k<l$.  However, if $i\larc{}l\in \mathrm{cflt}(\lambda)$, then there exists $i'\larc{}l'\in\lambda$ such that $i'=i<l<l'$ or $i'<i<l=l'$, contradicting $\nst_{j\slarc{}k}^\lambda=0$.  Thus, $\nst_{j\slarc{}k}^\nu= 0$. We can conclude that 
$$q^{\nst_{i\slarc{}j}^\lambda}-q^{\nst_{i\slarc{}j}^\nu}=0,$$
so our sum is zero if there exists $j\larc{}k\in \nu-(\lambda\cup\mathrm{cflt}(\lambda))$ such that $\nst_{j\slarc{}k}^\lambda=0$.

Suppose $\nu>\lambda$.  Then there exists $j\larc{}k\in \nu-\lambda$ maximal with respect to $k-j$.  If $\nst_{j\slarc{}k}^\lambda\neq 0$, then there would exist $i\larc{}l\in \lambda$ with $i<j<k<l$.  However, the maximality of our choice now contradicts $\lambda>\mu$.  Thus, $\nst_{j\slarc{}k}^\lambda= 0$, and our coefficient is 0, as desired.
\end{proof}

Furthermore, the nonzero coefficients are polynomials in $q$ with integer coefficients.

 \begin{corollary} For $\nu,\lambda\in \cS_n$,
$$\sum_{\mu\subseteq\nu} \chi^\lambda_\mu  \frac{(-1)^{|\nu-\mu|}}{q^{\nst_{\nu-\mu}^\nu}}\in \ZZ[q].$$
\end{corollary}
\begin{proof}
By Theorem \ref{CoefficientFormula}, it suffices to show that
$$\frac{q^{\dim(\lambda)}}{q^{|\lambda|+\mathrm{snst}_\nu^\lambda+\nst_\nu^\nu}}\in \ZZ[q].$$
Note that any arc $i\larc{}l$ can have at most $\lfloor\frac{l-i-1}{2}\rfloor$ arcs nested in it, so
$$\mathrm{snst}^{\lambda}_\nu\leq \sum_{i\slarc{}l\in \lambda} \lfloor\frac{l-i-1}{2}\rfloor\leq \frac{\dim(\lambda)-|\lambda|}{2} \qquad \text{and}\qquad \nst^{\nu}_\nu\leq \sum_{i\slarc{}l\in \nu} \lfloor\frac{l-i-1}{2}\rfloor\leq \frac{\dim(\nu)-|\nu|}{2}. $$
However, the coefficient is zero if $\nu>\lambda$, so 
$$\mathrm{snst}^{\lambda}_\nu +  \nst^{\nu}_\nu + |\lambda|\leq\frac{\dim(\lambda)-|\lambda|}{2} +  \frac{\dim(\nu)-|\nu|}{2} + |\lambda| \leq \dim(\lambda),$$
as desired.
\end{proof}

There are many specializations of Theorem \ref{CoefficientFormula}.   For example, as entries of a $|\cS_n|\times |\cS_n|$ matrix, we could consider the diagonal entries, as in the following corollary.

\begin{corollary} For $\lambda\in \cS_n$, 
$$\sum_{\mu\subseteq \lambda} \chi^\lambda_\mu  \frac{(-1)^{|\lambda-\mu|}}{q^{ \nst_{\lambda-\mu}^\lambda}}=(-1)^{|\lambda|}q^{\dim(\lambda)-\nst_\lambda^\lambda}.$$
\end{corollary}
\begin{proof}
This follows directly from Theorem \ref{CoefficientFormula}  and the observation that 
$$\prod_{i\slarc{}j\in\lambda\cap \nu} (q-1)q^{\nst_{i\slarc{}j}^\lambda}+q^{\nst_{i\slarc{}j}^\nu}=\prod_{i\slarc{}j\in\lambda} (q-1)q^{\nst_{i\slarc{}j}^\lambda}+q^{\nst_{i\slarc{}j}^\lambda}=q^{\nst_\lambda^\lambda+|\lambda|}.\qedhere$$
\end{proof}

%%%%%%%%%%%%%%%%%%%%%%%%%%%%%%%%%%%%%%%%%%%%%%%%%%%%%%%%%%%%%%%%%%
\section{Consequences}

One of the most immediate consequences of Theorem \ref{CoefficientFormula} is that the $\rho$ basis gives an $LU$-decomposition of the supercharacter table of $\UT_n(q)$ (That is, a product of an upper-triangular matrix  and a lower triangular matrix).

\begin{corollary}
The supercharacter table $C$ of $\UT_n(q)$ has a factorization
$$C=AB$$
where $A$ is a lower-triangular matrix with entries in $\ZZ[q]$ and $B$ is an upper-triangular with entries in $\ZZ[q^{-1}]$. 
\end{corollary}

We expect that interesting applications will come from such a result.  For now, we have a combinatorial formula for the determinant of the supercharacter table.

\begin{corollary}
The supercharacter table $C$ of $\UT_n(q)$ has a determinant
$$\det(C)=(-1)^{\sum_{\lambda\in \cS_n} |\lambda|} q^{\sum_{\lambda\in \cS_n} \dim(\lambda) -\nst_\lambda^\lambda}.$$
\end{corollary}

It is somewhat of a surprise that the sequences (which we have added to Sloane)
\begin{align*}
\dim(n)&=\sum_{\lambda\in \cS_n} \dim(\lambda)\qquad \qquad \text{[Sloane A200580]}\\
\mathrm{arcs}(n) & = \sum_{\lambda\in \cS_n} |\lambda| \qquad \qquad \qquad \text{[Sloane A200660]}\\
\nst(n) & = \sum_{\lambda\in \cS_n} \nst^\lambda_\lambda \qquad \qquad \quad\, \text{[Sloane A200673]}
\end{align*}
did not appear to be in the literature (or at least not in the Sloane integer sequences database).   However, the first two at least do appear to be related to the  recursive two-variable array (Sloane A011971), known as Aitken's array, given by
$$b[n,k]=\left\{\begin{array}{ll} \#\{\lambda\in \cS_n\mid k\larc{}n\in \lambda\} & \text{if $1\leq k<n$,}\\  \#\{\lambda\in \cS_n\mid j\larc{}n\notin \lambda, 1\leq j<n\}, & \text{if $k=n$.}\end{array}\right. $$
This sequence satisfies the recursion
\begin{align*}
b[1,1] & =1\\
b[n,1] & = b[n-1,n-1]\\
b[n,k] & =b[n,k-1]+b[n-1,k-1],
\end{align*}
and looks like
$$\begin{array}{c@{}c@{}c@{}c@{}c@{}c@{}c@{}c@{}c}  
&&  & & 1 & & && \\
& & & 1& & 2 &  &&\\
&&  2 & &  3 &  & 5 &&\\
& 5 && 7 && 10 && 15& \\
15 && 20 && 27&& 37 && 52
 \end{array}$$
 Note that the Bell numbers appear on the boundary of this triangle.  
 
\begin{proposition}
For $n\in \ZZ_{\geq 1},$
\begin{align*}
\mathrm{arcs}(n) &=  \sum_{k=1}^{n-1} k \cdot b[n,k]\\
\dim(n) &= \sum_{k=1}^{n-1} k (n-k) \cdot b[n,k].
\end{align*}
\end{proposition}
\begin{proof}
Consider first 
\begin{equation*}
\mathrm{arcs}(n) =  \sum_{\lambda\in \cS_n} |\lambda| =\sum_{1\leq i<j\leq n}  \#\{\lambda\in \cS_n\mid i\larc{}j\in \cS_n\}.
\end{equation*}
However,  
$$\#\{\lambda\in \cS_n\mid i\larc{}j\in \cS_n\} = b[n,n-j+i],$$
so
$$\mathrm{arcs}(n) =\sum_{1\leq i<j\leq n}  b[n,n-j+i]=\sum_{k=1}^{n-1} k b[n,k].$$
The argument for $\dim(n)$ is similar, but as we enumerate the arcs $i\larc{}j$, we add the statistic $j-i$.
\end{proof}

To understand the sequence $\nst(n)$ in a similar way, we need to define a slight variation of the sequence $b[n,k]$ given by
$$b[n,k,j]=\left\{\begin{array}{ll} 
\#\{\lambda\in \cS_n\mid j\larc{}n,k\larc{}n-1\in \lambda\} & \text{if $j<k<n-1$,}\\ 
 \#\{\lambda\in \cS_n\mid j\larc{}n\in \lambda,  i\larc{}n-1\notin \lambda, 1\leq i <n-1\} & \text{if $j<k=n-1$.}
\end{array}\right.$$
These numbers also satisfy a recursion given by
\begin{align*}
b[3,2,1] & = 1\\
b[n,2,1] &= b[n-1,n-2,1]\\
b[n,j+1,j]&=b[n,j+1,j-1]+b[n-1,j,j-1]\\
b[n,k,j] & = b[n,k-1,j]+b[n-1,k-1,j]
\end{align*}

\begin{proposition} For $n\in \ZZ_{\geq 1},$
$$\nst(n)=\sum_{j=1}^{n-3}\sum_{k=j+1}^{n-2}  j (k-j) b[n,k,j].$$
\end{proposition}

%%%%%%%%%%%%%%%%%%%%%%%%%%%%%%%%%%%%%%%%%%%%%%%%%%%%%%%%%%%%%%%%%%

\end{document}